\documentclass[12pt,leqno,fleqn]{amsart}
\usepackage{hyperref}
\usepackage{amsmath,amstext,amsthm,amssymb} 

\usepackage{float}
\usepackage{amsrefs}
\hypersetup{
    colorlinks=true,       
    linkcolor=blue,          
    citecolor=magenta,        
    filecolor=magenta,      
    urlcolor=cyan,          
}

\usepackage{graphicx}

\allowdisplaybreaks
\swapnumbers

\allowdisplaybreaks
\swapnumbers

\theoremstyle{plain} 
\newtheorem{lemma}[equation]{Lemma} 
\newtheorem{proposition}[equation]{Proposition} 
\newtheorem{theorem}[equation]{Theorem} 
\newtheorem{corollary}[equation]{Corollary} 
\newtheorem{conjecture}[equation]{Conjecture}
\newtheorem{question}[equation]{Question}

\theoremstyle{definition}
\newtheorem{definition}[equation]{Definition} 

\theoremstyle{remark}
\newtheorem{remark}[equation]{Remark}
\newenvironment{Pf}{\noindent{\textit{Proof of}}}{$\square$ }

\numberwithin{equation}{section}

\def\norm#1.#2.{\lVert#1\rVert_{#2}}
\def\Norm#1.#2.{\bigl\lVert#1\bigr\rVert_{#2}}
\def\NOrm#1.#2.{\Bigl\lVert#1\Bigr\rVert_{#2}}
\def\NORm#1.#2.{\biggl\lVert#1\biggr\rVert_{#2}}
\def\NORM#1.#2.{\Biggl\lVert#1\Biggr\rVert_{#2}}
\def\z{\zeta}
\def\D{{\mathbb D}}
\def\l{\lambda}
\def\N{{\mathbb N}}

\def\ip#1,#2,{\langle #1,#2\rangle}
\def\Ip#1,#2,{\langle#1,#2\rangle}
\def\IP#1,#2,{\langle#1,#2\rangle}
\def\E{{\mathcal E}}

\def\mid{\,:\,}

\def\XXint#1#2#3{{\setbox0=\hbox{$#1{#2#3}{\int}$}
     \vcenter{\hbox{$#2#3$}}\kern-.5\wd0}}

\newcommand{\T}{\mathbb{T}}
\newcommand{\TT}{\mathcal{T}}
\newcommand{\PP}{\mathcal{P}}
\newcommand{\LL}{\mathcal{L}}
\renewcommand{\E}{\mathbb{E}}

\begin{document}

\title[Sarason Conjecture on the Bergman Space]{Sarason Conjecture on the Bergman space}
\subjclass[2010]{Primary: Primary: 47B38, 30H20 Secondary: 42C40, 42A61,42A50 }
\keywords{Bergman spaces, Toeplitz products, two-weight inequalities, Bekoll\'e weights}

\author[A. Aleman]{Alexandru Aleman}
\thanks{}
\address{Centre for Mathematical Sciences, University of Lund, Lund, Sweden}
\email {aleman@maths.lth.se}

\author[S. Pott]{Sandra Pott}
\thanks{}
\address{Centre for Mathematical Sciences, University of Lund, Lund, Sweden}
\email {sandra@maths.lth.se}

\author[M.C. Reguera]{Maria Carmen Reguera}
\thanks{Supported partially by the Crafoord Foundation and by Lund University, Mathematics in the Faculty of Science, with a postdoctoral research grant}
\address{Department of Mathematics, Universitat Autonoma de Barcelona, Bellaterra (Barcelona), Spain}
\email {mreguera@mat.uab.cat}


\begin{abstract}
We provide a counterexample to the Sarason Conjecture  for the Bergman space and present a characterisation of bounded Toeplitz products on the Bergman space in terms of test functions by means of a dyadic model approach. We also present some results about two-weighted estimates for the Bergman projection. Finally, we introduce the class $B_\infty$ and give sharp estimates for the one-weighted Bergman projection.
\end{abstract}

\maketitle

\section{Introduction}
Let $dA$ denote Lebesgue area measure on the unit disc $\mathbb{D}$, normalized 
so that the measure of $\mathbb{D}$ equals 1.
The Bergman space $A^2(\D)$ is the closed subspace of analytic 
functions in the Hilbert space $L^2(\mathbb{D},dA)$. 
Likewise, the Hardy space $H^2(\T)$
 is the closed subspace of $L^2(\T)$  consisting of analytic functions.

The Bergman projection $P_B$, given by 
$$
    P_B f (z) = \int_\D\frac{f(\z)}{(1-\overline{\z}z)^2}dA(\z),
$$
 is the orthogonal projection from  
$L^2(\mathbb{D},dA)$
onto $A^2(\D)$, while the Riesz projection $P_R$ denotes the orthogonal projection from $L^2(\T)$ to $H^2(\T)$.
 For each  function $f \in L^2(\D)$ 
we have the densely defined  Bergman space Toeplitz 
operator $T_f$  on $A^2(\D)$, given by 
$$T_f u = P_B fu. $$
In the same way, given $f \in L^2(\T)$,  the Hardy space Toeplitz operator $\TT_f$ on $H^2$ is given by
$$
   \TT_f v = P_R fv,
$$
where $u$ and $v$ are suitable elements in $A^2$ and $H^2$, respectively.

For analytic $f$, it is easy to see
that both the Bergman space Toeplitz operator 
$T_f$  and the Hardy space Toeplitz operator $\TT_f$ are bounded, if and only if $f$ is a bounded function on 
$\mathbb{D}$.

In this paper, we shall  study the 
question as to which pairs of functions  $f, g \in A^2(\D)$ give rise to a bounded Toeplitz product operator
$$
     T_f T_{{g}}^* :   A^2(\D) \rightarrow A^2(\D).
$$
This questions has a rich history and interesting connections to Harmonic Analysis, as we outline below.

Sarason \cite{sarason} conjectured the following:

\begin{conjecture}[Sarason Conjecture for the Bergman space]
\label{c.sarasonB}
Let $f,g\in A^2(\D)$. Then $T_{f}T^{*}_{g}$ is bounded on  $A^2(\D)$, if and only if
\begin{equation}  \label{conj:sar1}
  b_{f,g}:=\sup_{z \in \D}    B(|f|^2)(z) B(|g|^2)(z) < \infty,
\end{equation}
where $B$ denotes the Berezin transform,
\begin{equation}
\label{e.berezin}
B f (z) = \int_\D\frac{f(\z)(1-|z|^{2})^{2}}{|1-\overline{\z}z|^4}dA(\z).
\end{equation}

\end{conjecture} 
 Likewise, he conjectured the following for the case of the
Hardy space:

\begin{conjecture}[Sarason Conjecture for the Hardy space]
\label{c.sarasonH}
Given $f,g\in H^2(\T)$,  $\TT_f \TT_g^*$ is bounded in $H^2(\T)$ if and only if 
\begin{equation}  \label{conj:sar2}
     \sup_{z \in \D}    \PP(|f|^2)(z) \PP(|g|^2)(z) < \infty,
\end{equation}
where $\PP$ denotes the Poisson extension.
\end{conjecture}

Both in the Bergman space and the Hardy space case, these questions are closely connected to
very interesting questions in Harmonic Analysis, namely two-weight estimates for the Bergman projection, respectively the Riesz projection.

Cruz-Uribe observed \cite{cruz} the following commutative diagram in the case of the Hardy space:
$$
\begin{matrix} 
&& \TT_f \TT_g^* && \\
     &  H^2(\T)   &     {\longrightarrow}     &   H^2(\T)& \\ &&&& \\
 M_{\bar g} & \rotatebox{90}{$\longleftarrow$}  &&  \rotatebox{90}{$\longrightarrow$}  
                                                                    & M_f   \\  && P_R && \\
     &  L^2(\frac{1}{|g|^2}, \T) & {\longrightarrow}& H^2(|f|^2, \T)& \\
\end{matrix}
$$

Here, $M_{\bar g}$, $M_f$ on the vertical sides denote multiplication with the respective 
symbols, and these operators are isometric by definition of the weights.
A similar argument can be made for the Bergman space,
\begin{equation}     \label{eq:bdiagram}
\begin{matrix} 
&& T_f T_g^* && \\
     &  A^2(\D)   &     {\longrightarrow}     &   A^2(\D)& \\ &&&& \\
 M_{\bar g} & \rotatebox{90}{$\longleftarrow$}  &&  \rotatebox{90}{$\longrightarrow$}  
                                                                    & M_f   \\  && P_B && \\
     &  L^2(\frac{1}{|g|^2}, \D) & {\longrightarrow}& A^2(|f|^2, \D)& \\
\end{matrix}
\end{equation}
again with isometric operators on the vertical sides.
One can thus see easily that the top row of each diagram is bounded, if and only if the bottom row is bounded.

Hence the question on the boundedness of Toeplitz products can be translated to the problem of boundedness of the two-weighted Bergman projection
\begin{equation}     \label{eq:berg}
               P_B: L^2(\D, \frac{1}{|g|^2}) \rightarrow    L^2(\D, |f|^2) 
\end{equation}
respectively boundedness of the two-weighted Riesz projection
\begin{equation}   
               P_R: L^2(\T, \frac{1}{|g|^2}) \rightarrow    L^2(\T, |f|^2) 
\end{equation}
in the case of the Hardy space.

This connection motivated  the Sarason conjectures \ref{c.sarasonB}, \ref{c.sarasonH} above.
 Namely, condition (\ref{conj:sar1}) is the natural two-weight form
of the B\'ekoll\'e-Bonami condition $B_2$ for a weight function $w$ on $\D$,  
$$
     \sup_{z \in \D} B(w)(z)    B(w^{-1})(z) < \infty,
$$
which is equivalent to the boundedness of the one-weighted Bergman projection 
\begin{equation}    \label{eq:bb}
               P_B: L^2(\D, w) \rightarrow    A^2(\D, w),
\end{equation}
and also to the boundedness of the maximal one-weighted Bergman projection
\begin{equation}   \label{eq:bbp}
               P^+_B: L^2(\D, w) \rightarrow    L^2(\D, w),
\end{equation}
where 
\begin{equation}
\label{e.maxBproj}
 P^+_B(f):= \int_\D\frac{f(\z)}{|1-\overline{\z}z|^2}dA(\z)
\end{equation}
(see \cite{MR497663}).
In the same way, (\ref{conj:sar2})  is the natural two-weight form
of the invariant Muckenhoupt  condition $A_2$ for a weight function $v$ ,
$$
     \sup_{z \in \D} \PP(v)(z)    \PP(v^{-1})(z) < \infty,
$$
which is equivalent to the boundedness of the one-weighted Riesz projection
\begin{equation}
               P_R: L^2(\D, v) \rightarrow    L^2(\D, v),
\end{equation}
or equivalently, the one-weighted Hilbert transform $H$
\cite{muckenhoupt}.

The problem of classifying those pairs of weights $(\rho,v)$ for which the two-weighted Riesz projection
      \begin{equation}  \label{eq:TwoWeightRiesz}
               P_R: L^2(\D, \rho) \rightarrow    L^2(\D, v),
\end{equation}
or equivalently, the two-weighted Hilbert transform is bounded, is a famous problem in Harmonic Analysis.
For a long time, it was conjectured that a version of  (\ref{conj:sar2}) for general weights $(\rho,v)$, the joint invariant $A_2$ condition
\begin{equation}
            \sup_{z \in \D} \PP(v)(z)    \PP(\rho^{-1})(z) < \infty
\end{equation}
characterises  (\ref{eq:TwoWeightRiesz}). 
This would in particular imply Sarason's conjecture on Hardy spaces. However, F. Nazarov disproved both this conjecture and the Sarason conjecture \ref{c.sarasonH}.
in 1997 \cite{nazarov}. The two-weight Hilbert transform problem, the problem of characterising boundedness of (\ref{eq:TwoWeightRiesz}), has been the subject of intense recent research
activity, see e.g. \cite{volberg}, \cite{1003.1596}, \cite{1001.4043}, \cite{1201.4319}, \cite{1301.4663} and the references therein.

Sarason's Conjecture \ref{c.sarasonB} for Toeplitz products on Bergman spaces, in contrast, has remained open till now.
The purpose of this paper is to provide a counterexample to this conjecture, depending on a new characterisation of bounded Toeplitz products on Bergman space by means of natural test function conditions. Our main results can be summarised as follows:

\begin{theorem}\label{t.conjfails} 
There exist functions $f,g\in L_a^2$ such that $b_{f,g}<\infty$, but $T_fT_g^*$ is not bounded on $L_a^2$.
\end{theorem} 

\begin{theorem}
\label{t.main}
Let $P^{+}_{B}(\cdot)$ be the maximal Bergman projection on the disc $\D$, and let $f, g\in A^{2}(\D)$. The following are equivalent
\begin{enumerate}
\item $T_{f}T^{*}_{g}: A^{2}(\D) \mapsto A^{2}(\D)$ is bounded;
\item $P_B( |g| \cdot): L^2(\D, |g|^2) \rightarrow  L^2(\D, |f|^2)$ bounded;
\item $P^+_B( |g| \cdot): L^2(\D, |g|^2) \rightarrow  L^2(\D, |f|^2)$ bounded;
\item \begin{enumerate}
\item $\displaystyle \norm |f|P^{+}_{B}(|g||g| 1_{Q_{I}}).L^{2}.\leq C_{0} \norm |g| 1_{Q_{I}}.L^{2}. , $
\item $\displaystyle\norm |g|P^{+}_{B}(|f||f| 1_{Q_{I}}).L^{2}.\leq C_{0} \norm |f| 1_{Q_{I}}.L^{2}. , $
\end{enumerate}
for all intervals $I\in \T$ and with constant $C_{0}$ uniform on $I$.
\end{enumerate}
\end{theorem}
Here, the first equivalence is Cruz-Uribe's observation, the second equivalence is proved in Section \ref{sec:pplus}, and the last equivalence 
is proved in Section \ref{sec:pos}. We will prove Theorem \ref{t.conjfails} in Section \ref{sec:counter}. 
Section \ref{sec:appl} is devoted to an application to the proof of sharp estimates for one-weighted Bergman projection.

Sufficient conditions close to Sarason's condition \ref{conj:sar1}  for the boundedness of Toeplitz products in the style of the so-called bump conditions can be found in  \cite{zheng} and in \cite{michal}.

In spite of the formal similarities of the Sarason conjectures in the Hardy space and in the Bergman space settings, the problem is quite different in both settings. 

Some aspects of the Bergman space setting are easier, because cancellation plays much less of a role in this setting, as already apparent from the equivalence of (\ref{eq:bb}) and (\ref{eq:bbp}).
To characterise boundedness of Toeplitz products, our strategy is thus to replace $P_B$ by $P_B^+$,  and to 
use two-weight techniques for positive operators  in Section \ref{sec:pos}, via a suitable dyadic model operator introduced in Section \ref{sec:dyadic}.
Somewhat surprisingly, it turns out that this is possible for the weights $\frac{1}{|g|^2}$, $|f|^2$ in (\ref{eq:bdiagram}). This is the equivalence of (2)
and (3) in Theorem \ref{t.main}, which will
be proved in Section \ref{sec:pplus}, and
 allows us to finally characterise the boundedness of Toeplitz products in Bergman space in terms of test function.

On the other hand, the special r\^{o}le played by weights coming from analytic functions, which we exploit in Section   \ref{sec:pplus} and which is in contrast to
the situation on the Hardy space, makes it much more difficult to find a counterexample of the Sarason Conjecture on Bergman space (\ref{conj:sar1}). We prove a
counterexample to the Sarason conjecture \ref{c.sarasonB}
in Section \ref{sec:counter}. For non-analytic symbols, or even one non-analytic symbol, such examples are much easier to find. In this case, the function 
$g$ in Lemma \ref{dynkin}, the construction of which forms the main part of the counterexample, can just be replace by $1 - |z|$.

\section{A dyadic model for the maximal Bergman projection}
\label{sec:dyadic}

In this section we aim to find a dyadic operator that models the behaviour of the maximal Bergman projection. To be precise, we find a dyadic averaging operator that is pointwise comparable to the maximal Bergman projection. 

The use of translations of a dyadic system to extend results from a dyadic setting to a continuous one is a well known tool. These ideas go back to the work of Garnett and Jones \cite{MR658065}, Christ \cite{MR951506} and also Tao Mei \cite{MR1993970}. In our case, we will use two of these dyadic systems to recover the maximal Bergman kernel from dyadic operators.

For $\beta\in \{0,1/3\}$, we define  $$\displaystyle \mathcal D^{\beta}:= \left\{ [2^{-j}2\pi m +2\pi\frac{1}{3},\,2^{-j}2\pi (m+1)+2\pi\frac{1}{3}) : m\in \mathbb N,\, j\in \mathbb N,\, j\geq 0, \, 0\leq m \leq 2^{j}  \right\}.$$ The key fact is that any interval in the torus is contained in one interval belonging to these two families of dyadic grids, moreover the measure of the two intervals is essentially the same. We formulate the result below. Its proof is a well-known exercise that the reader can find in many places, e.g. \cite{MR1993970}.

\begin{lemma}\label{l.grids}
Let $I$ be any interval in $\T$. Then there exists an interval $K\in \mathcal D^{\beta}$ for some $\beta\in \{0,1/3\}$ such that $I\subset K$ and $|K|\leq 6|I|$.
\end{lemma}

We define the family of dyadic operators that will control the maximal Bergman projection \eqref{e.maxBproj} as the following.

\begin{definition}\label{d.dyadicmodel}
Let $\mathcal D^{\beta}$ be one of the dyadic grids in $\T$ described above. For all $z,\xi \in \D$, we define the positive dyadic kernel
\begin{equation}\label{e.dyadicp}
K^{\beta}(z,\xi):=\sum_{I\in \mathcal D^{\beta}} \frac{1_{Q_{I}}(z)1_{Q_{I}}(\xi)}{|I|^{2}},
\end{equation}
where $Q_{I}$ is the Carleson box associated to $I$, namely
\begin{equation}
\label{e.carbox}
Q_{I}:=\{r\textup{e}^{i\theta}:\,\, 1-|I|\leq r<1 \text{ and } \, \textup{e}^{i\theta}\in I\},
\end{equation}
and $|I|$ stands for the normalized length of the interval. Associated to this kernel we define the following dyadic operator
\begin{equation}
\label{e.drm}
P^{\beta}f(z):=\sum_{I\in\mathcal D^{\beta}}\langle f, \frac{1_{Q_{I}}}{|I|^{2}}  \rangle 1_{Q_{I}}(z).
\end{equation}

\end{definition}

The following proposition proves the relation between the kernels \eqref{e.kernel} and the dyadic kernels described in \eqref{e.dyadicp}.

\begin{proposition}\label{p.pointwise}
There exist constants $C$ and $\tilde{C}$ such that for every $\beta_{0}\in \{0, 1/3\}$, every $f\in L^{1}_{loc}$ and $z\in \D$,
\begin{equation}
\label{e.dyadicmp}
\tilde{C} P^{\beta_{0}}f(z)\leq P^{+}_{B}f(z)\leq C\sum_{\beta\in  \{0,1/3\}}P^{\beta}f(z),
\end{equation}
where $P_{B}^{+}$ is the maximal Bergman projection as defined in \eqref{e.maxBproj} and $P^{\beta}$ the dyadic operator described in \eqref{e.drm}.

\end{proposition}

\begin{proof}[ Proof of Proposition \eqref{p.pointwise}]
Let $K(z,\xi)$ denote the kernel associated to the maximal Bergman projection, i.e.,

\begin{equation}
\label{e.kernel}
K(z,\xi)= \frac{1}{|1-z\bar{\xi}|^{2}}.
\end{equation}
Then it is enough to prove that there exist constants $C$ and $\tilde{C}$ such that for every $\beta_{0}$ and every $z, \xi$ in $\D$ we have the following estimates on the kernel,

\begin{equation}
\label{e.pointwisek}
\tilde{C} K^{\beta_{0}}(z,\xi) \leq  K(z,\xi) \leq C\sum_{\beta\in  \{0,1/3\}} K^{\beta}(z,\xi)
\end{equation}

Let us first prove the left hand side of \eqref{e.pointwisek}. We consider $z=r_{0}\textup{e}^{i\theta_{0}}$ and $\xi=s_{0}\textup{e}^{i\varphi_{0}}$. 
Without loss, we can assume that $r_0 \le s_0$. 
We choose $I_{0}\in\mathcal D^{\beta_{0}}$ to be the minimal interval such that $|I_{0}|\geq 1-r_{0}$ and $\textup{e}^{i\theta_{0}},\, \textup{e}^{i\varphi_{0}}\in I_{0}$. Then, it is easy to see that $z,\xi\in Q_{I_{0}}$. It could be that such an interval doesn't exist, in that case the inequality is trivially true.  From  $z,\xi\in Q_{I_{0}}$ we can deduce
\begin{equation}
\label{e.pint}
\sum_{I\in \mathcal D^{\beta}} \frac{1_{Q_{I}(z)}1_{Q_{I}(\xi)}}{|I|^{2}}=\sum_{I, I_{0}\subset I}\frac{1}{|I|^{2}} \leq C \frac{1}{|I_{0}|^{2}}.
\end{equation}
To conclude the proof of the left hand side, we need to show
\begin{equation}
\label{e.nts}
|1-z\bar{\xi}|^{2}\leq C|I_{0}|^{2},
\end{equation}
for some uniform constant $C$.
We can write $|1-z\bar{\xi}|^{2}$ as 
\begin{equation}     \label{eq:angle}
|1-z\bar{\xi}|^{2}=(1-r_{0}s_{0})^{2}+4r_{0}s_{0}\sin^{2}(\frac{\theta_{0}-\varphi_{0}}{2}).
\end{equation}

We distinguish two cases, when $(1-r_{0}s_{0})^{2}$ is the majorant term, and when $4r_{0}s_{0}\sin^{2}(\frac{\theta_{0}-\varphi_{0}}{2})$ is the majorant.

\begin{enumerate} 
\item Case 1. If $(1-r_{0}s_{0})^{2}>4r_{0}s_{0}\sin^{2}(\frac{\theta_{0}-\varphi_{0}}{2})$,
then 
$$
|1-z\bar{\xi}|^{2}\leq 2(1-r_{0}s_{0})^{2}\leq 8|I_{0}|^{2},
$$
as desired.

\item Case 2. Suppose on the contrary that $(1-r_{0}s_{0})^{2}\leq 4r_{0}s_{0}\sin^{2}(\frac{\theta_{0}-\varphi_{0}}{2})$.
Since $\textup{e}^{i\theta_{0}},\, \textup{e}^{i\varphi_{0}}\in I_{0}$, we know that $|I_{0}|\geq |\theta_{0}-\varphi_{0}|$. Then
$$
|1-z\bar{\xi}|^{2}\leq 8r_{0}s_{0}\sin^{2}(\frac{\theta_{0}-\varphi_{0}}{2})\leq 2|\theta_{0}-\varphi_{0}|^{2}\leq 2|I_{0}|^{2},
$$
as desired. 

\end{enumerate}
Therefore we have concluded the proof of the left hand side of \eqref{e.pointwisek}. We now turn to the right hand inequality in \eqref{e.pointwisek}. Once again let us fix $z,\xi \in \D$, and write them as before as
$z=r_{0}\textup{e}^{i\theta_{0}}$ and $\xi=s_{0}\textup{e}^{i\varphi_{0}}$.  

It is enough to prove the existence of an interval $I_{0}$ in $\T$ such that $z,\xi\in Q_{I_{0}}$  and $|I_{0}|^{2}\simeq |1-z\bar{\xi}|^{2}$. If such an interval exists, by Lemma $\ref{l.grids}$, we find $K\in \mathcal D^{\beta}$ for some $\beta\in \{0, 1/3\}$ such that $I_{0}\subset K$ and $|K|\leq 6|I_{0}|$. Now the proof of the proposition follows from the set of inequalities below:

\begin{eqnarray*}
\frac{1}{|1-z\bar{\xi}|^{2}}& \lesssim & \frac{1}{|I_{0}|^{2}}\\
 & \lesssim &\frac{1}{36|K|^{2}} \\
 & \leq & C \sum_{\substack{I\in  \mathcal D^{\beta}\\ K\subset I}}\frac{1_{Q_{I}}(z)1_{Q_{I}}(\xi)}{|I|^{2}}\\
 & \leq & C  \sum_{\beta\in \{0,1/3\}} K^{\beta}(z,\xi).
\end{eqnarray*}

Thus we have reduced the problem to prove the existence $I_{0}$ interval in $\T$ such that $z,\xi\in Q_{I_{0}}$  and $|I_{0}|^{2}\simeq  |1-z\bar{\xi}|^{2}$. Notice 
than in the normalized arc measure $|\cdot|$, we will always have $|\theta_{0}-\varphi_{0}|\leq 1/2$, and since $|\sin x|\simeq  |x|$ for $|x|\leq \pi/2$, we have

\begin{equation}
\label{e.eqnorm}
|1-z\bar{\xi}|^{2}\simeq  (1-r_{0}s_{0})^{2}+r_{0}s_{0}|\theta_{0}-\varphi_{0}|^{2} \simeq     (1-r_{0}^2)^{2}+ |\theta_{0}-\varphi_{0}|^{2}                              .
\end{equation}
by (\ref{eq:angle}).
Let us choose $I_{0}$ to be a minimal interval such that
$$
|I_{0}|^{2}= \max((1-r_{0}^2)^{2},\, |\theta_{0}-\varphi_{0}|^{2})
$$
and $\textup{e}^{i\theta_{0}}, \textup{e}^{i\varphi_{0}}\in I_{0}$. It is easy to see that $z,\xi\in Q_{I_{0}}$. We have to  prove that $|I_{0}|^{2}\simeq  |1-z\bar{\xi}|^{2}$. But this follows directly from 
(\ref{e.eqnorm}).

This finishes the proof of the proposition.

\end{proof}


\section{Two weight estimates for the Maximal Bergman Projection}
\label{sec:pos}

In this section we establish two-weight estimates for the maximal Bergman projection. We start by providing a two-weight characterization of boundedness for general dyadic positive operators, to conclude the desired estimates for the maximal Bergman projection as a consequence of the dyadic result and inequalities  \eqref{e.dyadicmp}.

There are three equivalent formulations for two weighted inequalities that we will use in turn. A weight function will be an nonnegative measurable function on 
$\mathbb{R^n}$, not necessarily locally integrable. 
Let $w,v$ be weight functions in $\mathbb R^{n}$, let $1<p<\infty$ and $p'$ its dual exponent. We define $\sigma:=v^{1-p'}$, which is usually called the dual weight of $v$. Let $T$ be an operator. Then the following are equivalent:

\begin{equation} \label{e.oldtwo}
T:\, L^{p}(v)\mapsto L^{p}(w)
\end{equation}

\begin{equation} \label{e.standtwo}
T(\sigma\cdot):\, L^{p}(\sigma)\mapsto L^{p}(w)
\end{equation}

\begin{equation} \label{e.aletwo}
w^{1/p}T(\sigma^{1/p'}\cdot):\, L^{p}\mapsto L^{p}.
\end{equation}

In this section we will mostly use \eqref{e.standtwo}, although for the Sarason problem, \eqref{e.aletwo} is more natural and will frequently appear.

Throughout this section, we will denote the expectation of a function $f$ over a cube $Q$ by
$$
\mathbb E_{Q}|f|,
$$
and the expectation of a function $f$ over a cube $Q$ with respect to a weight $\sigma$ will be denoted by
$$
\mathbb E^{\sigma}_{Q}|f|.
$$
We consider a dyadic grid in $\mathbb R^{n}$ and denote it by $\mathcal D$.  The class of operators we are interested in are dyadic positive operators of the form
\begin{equation}
\label{e.dyadop}
T(f):=\sum_{Q\in \mathcal D}\tau_{Q} (\mathbb E_{Q}|f|)1_{Q},
\end{equation}
where $\tau_{Q}$ is a sequence of nonnegative scalars and $1_{E}$ indicates the characteristic function on the set $E$.

Given two weights $w$ and $\sigma$, we aim to characterise the boundedness of the operator $T$ in the two-weight setting. More precisely,
 we state the question as follows:

\begin{question}\label{q.twow}
Characterize the pairs of weights $w$ and $\sigma$ for which
\begin{equation}
\label{e.question}
T(w\cdot)\mid L^{p}(w)\mapsto  L^{p}(\sigma)   \text{ is bounded.}
\end{equation}
\end{question}

The following theorem provides an answer to this question. In this precise form, it is due to Lacey, Sawyer and Uriarte-Tuero \cite{0911.3437}. 
We present a simplified version of their original proof. Our proof can also be adapted to the disc, with the Carleson cubes associated to a dyadic grid in $\T$ as the dyadic family.

\begin{theorem}
\label{t.maind}
Let $w,\sigma$ be two weights and let $T$ be a dyadic positive operators as in \eqref{e.dyadop}. Then 
\begin{equation}    \label{eq:two-weight-pos}
T(w\cdot)\mid L^{p}(w)\mapsto  L^{p}(\sigma)   
\end{equation}
is bounded, if and only if
  \begin{equation} 
 \label{e.test}  \norm T(w 1_{Q}).L^{p}(\sigma).^{p}\leq C_{0} w(Q), 
 \end{equation}
and 
\begin{equation}
\label{e.testd} 
\norm T^{*}(\sigma 1_{Q}).L^{p}(w).^{p}\leq C_{0}^{*}\sigma(Q), 
\end{equation}

for all $Q$ dyadic cube in $\mathcal D$, and constants $C_{0}$ and $C_{0}^{*}$ independent of the cubes $Q$. Moreover, there exists a constant $c$
independent  of  $T$ and $w$, $\sigma$, such that
$$
\|T(w\cdot)\|_{L^{p}(w)\to  L^{p}(\sigma)}   \le c (C_0 + C_0^*).
$$
\end{theorem}

\begin{remark}
\label{r.twdpos}
In fact, one needs only weaker testing conditions in order to get boundedness of the operator, namely,  \eqref{e.test} and \eqref{e.testd} can be replaced by 
\begin{equation}
\label{e.stest}  \norm T_{in,Q}(w 1_{Q}).L^{p}(\sigma).^{p}\leq C_{0} w(Q), 
\end{equation}
and
\begin{equation}
\label{e.stestd} \norm T^{*}_{in,Q}(\sigma 1_{Q}).L^{p}(w).^{p}\leq C_{0}^{*}\sigma(Q)
\end{equation}
respectively, where $\displaystyle T_{in,Q}:=\sum_{\substack{P\in \mathcal D\\ P\subset Q  }}\tau_{P} (\mathbb E_{P}|f|)1_{P}$. 
The use of these weaker testing conditions \eqref{e.stest} and \eqref{e.stestd} can be traced in the proof of Theorem \ref{t.maind} below.
\end{remark}
The characterization in terms of testing conditions for dyadic positive operators in $\mathbb R^{n}$ was provided by Lacey, Sawyer and Uriarte-Tuero \cite{0911.3437}, based on previous work of Eric Sawyer in the continuos case \cites{MR676801,MR930072}. In a recent paper \cite{1201.1455}, Treil was able to simplify their argument. Our contribution aims to further simplify Treil's latest proof, we use one discretizing procedure (the Corona decomposition in subsection 3.1), and we avoid the appeal to 
the Carleson Embedding theorem.

Let $\Delta w$ denote the weight obtained from $w$ by averaging
$$
     \Delta w = \sum_{\substack{I \text{ interval}\\ I\subset \T}    } \E_{T_I} (w 1_{T_I}) 1_{T_I}     
$$
\begin{remark}
The boundedness of 
$$
T(w\cdot)\mid L^{p}(w)\mapsto  L^{p}(\sigma)
$$ 
depends only $\Delta w$ and $\Delta \sigma$.
\end{remark}

\subsection{A Corona Decomposition}

For now and throughout this section, we will assume without loss of generality that the function $f$ is positive. 

\begin{definition} \label{d.corona}
Let $Q_{0}$ be a cube in $\mathcal D$ and let  $\mathcal D_{0}$ be a family of cubes contained in $Q_{0}$.
 Let $w$ be a weight in $\mathbb R^{n}$ and let $f$ be a positive locally integrable function. We define 
$$
\mathcal L(Q_{0})= \{ Q\in \mathcal D_{0}\,:\,  Q \text{ is a maximal cube in $\mathcal{D}_0$ such that } \mathbb E^{w}_{Q}|f| 
> 4\mathbb E^{w}_{Q_{0}}|f|\}.
$$
We define 
$$
\mathcal L_0:= \{Q_{0}\}
$$
recursively
$$
\mathcal L_{i}:= \cup_{L\in\mathcal L_{i-1} }\mathcal L(L).
$$
We will denote  the union of all the stopping cubes by $\mathcal L:= \cup_{i\geq 0} \mathcal L_{i}$.
 We notice that we could also define the starting family $\mathcal L_0$ as a union of disjoint maximal cubes and repeat the above construction in each one of the cubes in $\mathcal L_0$.
Given $Q\in \mathcal D_0$, we define $\lambda(Q)$ as the minimal cube $L\in \mathcal L$ such that $ Q\subset L$ and $\mathcal D(L):=\{Q\in \mathcal D_{0}\,:\, \lambda(Q)=L\}$.
\end{definition}

We consider now the dyadic Hardy-Littlewood maximal function in its weighted form. For a weight $w$, we define
\begin{equation}
\label{e.max}
M_{w}f(x)= \sup_{Q\in \mathcal D} \frac{1_{Q}}{w(Q)}\int_{Q} |f|wdm,
\end{equation}
where $dm$ stands for the Lebesgue measure in $\mathbb R^{n}$.
The following result is a well-known classical theorem.
\begin{theorem}
\label{t.mbounded}
\begin{equation}
\label{e.mbounded}
\norm M_{w}f. L^{p}(wdm).\leq C \norm f.L^{p}(wdm).
\end{equation}
where the constant $C$ is independent of the weight $w$.
\end{theorem}

The stopping cubes in Definition \ref{d.corona} provide the right collection of sets to linearise the dyadic Hardy Littlewood maximal function described in \eqref{e.max}, i.e., we have the following pointwise estimate:

\begin{equation}
\label{e.pointwise}
\sum_{\mathcal L} (\mathbb E_{L}^{w}|f|) 1_{L}(x) \lesssim M_{w}f(x) \quad \text{ for all } x\in \mathbb R^{n}.
\end{equation}

The proof of \eqref{e.pointwise} is an exercise. Suppose $x$ is not contained in any of the cubes of the starting collection $\mathcal L_{0}$, then the left hand side is zero, and
the inequality is trivially true. If, on the contrary, 
$x\in Q_{0}$ for some  $Q_{0}\in\mathcal L_{0}$, there exists a stopping cube $L' \in \LL$ with minimal side length such that $x\in L'$. We also know that the expectations are increasing geometrically, i.e.,
$$
\mathbb E_{L}^{w}|f|>4\mathbb E_{\tilde{L}}^{w}|f|, \text{ for all } L,\tilde{L}\in \mathcal L, \, \tilde{L}\subsetneq L,
$$
therefore 
$$
\sum_{\substack{L\in \mathcal L\\ x\in L}} \mathbb E_{L}^{w}|f|\lesssim E_{L'}^{w}|f| \leq M_{w}f(x),
$$
concluding the proof of \eqref{e.pointwise}.

An application of \eqref{e.pointwise} and Theorem \ref{t.mbounded} provides the following useful inequality:

\begin{equation}
\label{e.lbounded}
\sum_{L\in \mathcal L} (\mathbb E_{L}^{w}|f|)^{p} w(L)\lesssim \norm f.L^{p}(w dm).
\end{equation}

\subsection{Proof of Theorem \ref{t.maind}}

\begin{proof}

We are now ready to prove the main theorem in this section. We will assume there is a finite collection of dyadic cubes $\mathcal Q$ in the definition of the operator $T$, and we will prove the operator norm is independent of the chosen collection. So from now on 
$$
Tf=\sum_{Q\in \mathcal Q} \tau_{Q}(\mathbb E_{Q}f) 1_{Q}
$$

It is enough to prove boundedness of the bilinear form  $\langle T(wf), g\sigma \rangle$, where $0\leq f\in L^{p}(w)$ and $0\leq g\in L^{p'}(\sigma)$. 
Following the argument in \cite{1201.1455}, we seek an estimate of the form
\begin{equation}
\label{e.goal}
\langle T(wf), g\sigma \rangle \leq A\norm f.L^{p}(w).\norm g.L^{p'}(\sigma). + B\norm f.L^{p}(w).^{p}.
\end{equation}

We first divide the cubes in $\mathcal Q$ into two collections $\mathcal Q_{1}$ and $\mathcal Q_{2}$ according to the following criterion. A cube $Q$ will belong to $\mathcal{Q}_{1}$, if
\begin{equation}
\label{e.crit}
\left( \mathbb E_{Q}^{w}f\right)^{p}w(Q)\geq \left( \mathbb E_{Q}^{\sigma}g\right)^{p'}\sigma(Q),
\end{equation}
and it will belong to $\mathcal{Q}_{2}$ otherwise. This reorganisation of the cubes allows us to write $T=T_{1}+T_{2}$, where
$$
T_{i}f=\sum_{Q\in \mathcal Q_{i}} \tau_{Q}(\mathbb E_{Q}f )1_{Q}, \,\, i=1,2.
$$

The idea of writing $T$ as the sum of $T_1$ and $T_2$ was already present in the work of Treil \cite{1201.1455} and previously in the work of Nazarov, Treil and Volberg \cite{MR1685781}.

We will prove boundedness of $T_{1}$ using the testing condition \eqref{e.test}. The boundedness of $T_{2}$ can be proven analogously to $T_1$, only using \eqref{e.testd} this time.

\begin{eqnarray*}
\langle T_{1}(wf), g\sigma \rangle &=&\sum_{Q\in \mathcal Q_{1}} \tau_{Q}\mathbb E_{Q}(fw)\ip g\sigma, 1_{Q},\\
&=& \sum_{L\in \mathcal L}\sum_{Q\in \mathcal D(L)}  \tau_{Q}\mathbb E_{Q}(fw)\ip g\sigma, 1_{Q},\\
&=& \sum_{L\in \mathcal L} \langle T_{L}(wf), g\sigma \rangle,
\end{eqnarray*}
where $\LL$ is a collection of stopping cubes in the family $\mathcal{Q}_1$, to be specified below, and
 $T_{L}f= \sum_{Q\in \mathcal D(L)}  \tau_{Q}\mathbb E_{Q}(f)1_{Q}$. 
 To find the collection of stopping cubes $\LL$,
  we define $\mathcal L_{0}$ as the collection of  maximal cubes in the family $\mathcal Q_{1}$,  and follow the Definition \ref{d.corona} for given $f$ and $w$ to define $\mathcal L$, with $\mathcal(Q_1) $ as
  our family of dyadic cubes.

We are going to estimate the bilinear form 
\begin{equation}
\label{e.bil}
\sum_{L\in\mathcal L}\langle T_{L}(wf), g\sigma \rangle,
\end{equation}
but before doing this, let us look at the norm of $T_{L}$. We claim that
\begin{equation}
\label{e.normtl}
\norm T_{L}(wf). L^{p}(\sigma).^{p}\leq C_{0} 4^p  \left(\mathbb E_{L}^{w}(f)\right)^{p}w(L).
\end{equation}

This is easily verified by
\begin{eqnarray*}
\norm T_{L}(wf). L^{p}(\sigma).^{p}&=& \norm  \sum_{Q\in \mathcal D(L)}  \tau_{Q}\mathbb E_{Q}(fw)1_{Q}. L^{p}(\sigma).^{p}\\
&=& \norm  \sum_{Q\in \mathcal D(L)} \frac{w(Q)}{|Q|} \tau_{Q}\mathbb E_{Q}^{w}(f)1_{Q}. L^{p}(\sigma).^{p}\\
&\leq & 4^p\left(\mathbb E_{L}^{w}(f)\right)^{p} \norm  \sum_{Q\in \mathcal D(L)} \frac{w(Q)}{|Q|} \tau_{Q}1_{Q}. L^{p}(\sigma).^{p}\\
&\leq & 4^p\left(\mathbb E_{L}^{w}(f)\right)^{p} \norm  T(w1_{Q}). L^{p}(\sigma).^{p}\\
&\leq & 4^pC_{0}  \left(\mathbb E_{L}^{w}(f)\right)^{p}w(L),
\end{eqnarray*}
where in the first inequality we have used that $Q\in \mathcal D(L)$ are not stopping cubes, and in the last inequality, the testing condition \eqref{e.test}.

We now estimate \eqref{e.bil}. 
\begin{equation}
\label{e.nose}
\sum_{L\in\mathcal L} \langle T_{L}(wf), g\sigma \rangle= \sum_{L\in\mathcal L}\int T_{L}(wf)(x) g(x)\sigma(x)dx= (I)+(II), 
\end{equation}

where
$$
(I)=\sum_{i}\sum_{L\in\mathcal L_{i}} \int_{\displaystyle L\setminus \cup_{\substack{L'\in \mathcal L_{i+1}\\L'\subset L}}L'}T_{L}(wf)(x) g(x)\sigma(x)dx,
$$
and 
$$
(II)=\sum_{i}\sum_{L\in\mathcal L_{i}} \int_{\displaystyle \cup_{\substack{L'\in \mathcal L_{i+1}\\L'\subset L}}L'}T_{L}(wf)(x) g(x)\sigma(x)dx.
$$

We proceed to estimate $(I)$,

\begin{align*}
(I)&\leq  \sum_{i}\sum_{L\in\mathcal L_{i}} \norm T_{L}(fw). L^{p}(\sigma). \norm g1_{L\setminus \cup_{\substack{L'\in \mathcal L_{i+1}\\L'\subset L}}L'}.L^{p'}(\sigma).\\
 & \leq \left( \sum_{i}\sum_{L\in\mathcal L_{i}} \norm T_{L}(fw). L^{p}(\sigma).^{p}  \right)^{1/p} \left( \sum_{i}\sum_{L\in\mathcal L_{i}} \norm g1_{L\setminus \cup_{\substack{L'\in \mathcal L_{i+1}\\L'\subset L}}L'}.L^{p'}(\sigma).^{p'}\right)^{1/p'}\\
 &\leq  4 C_{0}^{1/p}\left ( \sum_{L\in \mathcal L} \left(\mathbb E_{L}^{w}f\right)^{p} w(L)\right)^{1/p}\norm g.L^{p'}(\sigma).\\
 &\lesssim C_{0}\norm f.L^{p}(w). \norm g.L^{p'}(\sigma).,
\end{align*}
where in the first two inequalities, we have used H{\"o}lder's inequality, and in the third one we have used the testing condition \eqref{e.test} and the fact that $\displaystyle \cup_{i}\cup_{L\in \mathcal L_{i}}L\setminus \cup_{\substack{L'\in \mathcal L_{i+1}\\L'\subset L}} L'$ forms a partition of the maximal cubes in $\mathcal L_{0}$. For the last inequality, we have used \eqref{e.lbounded}.

We now turn to $(II)$. Before we proceed with the estimate, let us note the following remark.
\begin{remark} \label{r.constant} Let $L\in \mathcal L$ be fixed,  then the operator $T_{L}(fw)$ is constant on $L'$, where $L'\in \mathcal L, \, L'\subsetneq L$. We will denote this constant by $T_{L}(fw)(L')$.
\end{remark}
 
Taking this remark into account, we get the following estimates for fixed $L\in \mathcal L_{i}$:

\begin{eqnarray*}
\int_{ \cup_{\substack{L'\in \mathcal L_{i+1}\\L'\subset L}}L'}T_{L}(wf)(x) g(x)\sigma(x)dx & = & \sum_{\substack{L'\in \mathcal L_{i+1}\\L'\subset L}} T_{L}(fw)(L')\int_{L'}g\sigma dx\\
&=& \sum_{\substack{L'\in \mathcal L_{i+1}\\L'\subset L}} \int_{L'} T_{L}(fw)(x) \left( \mathbb E_{L'}^{\sigma}g\right)\sigma(x) dx\\
&=& \int_{L} T_{L}(fw)(x) \left(\sum_{\substack{L'\in \mathcal L_{i+1}\\L'\subset L}} \mathbb E_{L'}^{\sigma}g 1_{L'}(x)\right) \sigma(x) dx\\
&\leq & \left\|   T_{L}(fw)\right\|_{L^{p}(\sigma)  }    \left\|      \sum_{\substack{L'\in \mathcal L_{i+1}\\ L'\subset L}} \mathbb E_{L'}^{\sigma}g 1_{L'} \right\|_{L^{p'}(\sigma)}\\
&=& \norm T_{L}(fw).L^{p}(\sigma). \left( \sum_{\substack{L'\in \mathcal L_{i+1}\\L'\subset L}} (\mathbb E_{L'}^{\sigma}g)^{p'}\sigma(L')  \right)^{1/p'}\\
&\leq & 4 C_0\mathbb E_{L}^{w}|f| w(L)^{1/p}\left( \sum_{\substack{L'\in \mathcal L_{i+1}\\L'\subset L}} (\mathbb E_{L'}^{w}f)^{p}w(L') \right)^{1/p'},
\end{eqnarray*}

where we have used Remark \ref{r.constant}, H{\"o}lder's inequality, \eqref{e.normtl} and the hypothesis \eqref{e.crit}. We now proceed to sum the previous estimates in $L$  to obtain the desired bound for $(II)$.

\begin{eqnarray*}
(II)&\lesssim & \sum_{i}\sum_{L\in\mathcal L_{i}} \mathbb E_{L}^{w}|f| w(L)^{1/p}\left( \sum_{\substack{L'\in \mathcal L_{i+1}\\L'\subset L}} (\mathbb E_{L'}^{w}f)^{p}w(L') \right)^{1/p'}\\
&\lesssim & \left( \sum_{L\in \mathcal L} (\mathbb E_{L}^{w}|f|)^{p}w(L)\right)^{1/p} \left( \sum_{i}\sum_{L\in\mathcal L_{i}}\sum_{\substack{L'\in \mathcal L_{i+1}\\L'\subset L}} (\mathbb E_{L'}^{w}f)^{p}w(L') \right)^{1/p'}\\
&\lesssim &\norm f.L^{p}(w). \norm f.L^{p}(w).^{p/p'} \lesssim \norm f.L^{p}(w).^{p}.
\end{eqnarray*}

Adding $(I)$ and $(II)$, we get the desired estimate \eqref{e.goal}.

\end{proof}

We now turn to the two weight characterization for the case of the maximal Bergman projection $P^{+}_{B}$ and its associated dyadic model $P^{\beta}$. We start with $P^{\beta}$. One can state the following theorem.

\begin{theorem}
\label{c.twodB}
Let $\mathcal D^{\beta}$ be a fixed dyadic grid in $\T$ and let $P^{\beta}$ as defined in \eqref{e.drm}. Then 
$$
P^{\beta}(w\cdot)\mid L^{p}(w)\rightarrow  L^{p}(\sigma)
$$
is bounded, if and only if
 
 \begin{equation} 
 \label{e.testpbeta}  \norm \sum_{\substack{I\in \mathcal D^{\beta}\\ I\subset I_{0}}}\langle w 1_{Q_{I_{0}}}, \frac{1_{Q_{I}}}{|I|^{2}}\rangle 1_{Q_{I}} .L^{p}(\sigma).^{p}\leq C_{0} w(Q_{I_{0}}), 
 \end{equation}
and 
\begin{equation}
\label{e.testpbetad}  \norm \sum_{\substack{I\in \mathcal D^{\beta}\\ I\subset I_{0}}}\langle \sigma 1_{Q_{I_{0}}}, \frac{1_{Q_{I}}}{|I|^{2}}\rangle 1_{Q_{I}} .L^{p'}(w).^{p'}\leq C_{0}^* \sigma(Q_{I_{0}}), 
\end{equation}
for all $I$ dyadic interval in $\mathcal D^{\beta}$, where $Q_{I}$ represents the Carleson box associated to $I$ and the constants $C_{0}$ and $C_{0}^{*}$ are independent of the intervals $I$. 
Moreover, there exists a constant $c>0$ independent of the weights, such that
$$
      \left\|  P^{\beta}(w\cdot)  \right\|_{L^p(w)^p \to L^p(\sigma)} \le c(C_0 + C_0^*).
$$
\end{theorem}

The proof of Theorem \ref{c.twodB} in the disc $\D$ is identical to the one we describe in Theorem \ref{t.maind}. In the case of the disc, our dyadic system will be described by the Carleson cubes associated to the intervals in the dyadic grid $\mathcal D^{\beta}$ in $\T$. The boundedness of the weighted Hardy-Littlewood maximal function over dyadic Carleson cubes in the disc will be used instead of Theorem \ref{t.mbounded}. We will also consider the testing conditions \eqref{e.stest} and \eqref{e.stestd}. The details of the proof are left to the reader.

We obtain the following corollary, which presents a two weight characterization for the maximal Bergman projection.

\begin{corollary}
\label{c.twmBp} 
Let $P^{+}_{B}$ be the maximal Bergman projection in the disc $\D$, let $1<p<\infty$ and $p'$ its dual exponent and let $w, \sigma$ be two weight functions. Then $$M_{w^{1/p}}P^{+}_{B}M_{\sigma^{1/p'}}:\, L^{p}(\D) 
\rightarrow L^{p}(\D)$$
is bounded, if and only if
\begin{equation}
\label{e.testmBp}
\norm M_{w^{1/p}}P^{+}_{B}M_{\sigma^{1/p'}}(1_{Q_{I}}\sigma^{1/p}).L^{p}(\D).\leq C_{0} \norm 1_{Q_{I}}\sigma^{1/p}. L^{p}(\D).
\end{equation}
and
\begin{equation}
\label{e.testmBpd}
\norm M_{\sigma^{1/p'}}P^{+}_{B}M_{w^{1/p}}(1_{Q_{I}}w^{1/p'}).L^{p'}(\D).\leq C_{0}^{*} \norm 1_{Q_{I}}w^{1/p'}. L^{p'}(\D).,
\end{equation}
for any interval $I$ in $\T$, where the constants $C_{0}$ and $C_{0}^{*}$ are independent of the choice of interval. 

Moreover, there exists a constant $c>0$ independent of the weights, such that
$$
      \left\|M_{w^{1/p}}P^{+}_{B}M_{\sigma^{1/p'}}  \right\|_{L^2 \to L^2} \le c(C_0 + C_0^*).
$$
\end{corollary}

As in the introduction, the operators $M_{h}$ stand for the operator of multiplication by the symbol $h$. 
\proof
We only have to prove one direction. By  the first inequality in  \eqref{e.dyadicmp}, the testing condition (\ref{e.testmBp}) and (\ref{e.testmBpd}) imply the corresponding testing condition for each $P^\beta$, and
therefore the uniform boundedness of all $P_\beta$ by Theorem \ref{c.twodB}. 
The second inequality in  \eqref{e.dyadicmp} now implies the boundedness of $M_{w^{1/p}}P^{+}_{B}M_{\sigma^{1/p'}}$ with the required norm bounds.
\qed

We note that the positivity of $P^{+}_{B}$ and the left hand-side of \eqref{e.dyadicmp} are crucial here to recover the non-dyadic case from the dyadic one.
 This advantage is not present in the case of cancellative operators such as the Bergman projection itself.


\section{$P$ and $P_+$ are equivalent}

\label{sec:pplus}

Given  $f,g\in L^2(\D)$, we denote as before
$$b_{f,g}=\sup_{z\in \D}B^{1/2}(|f|^2)(z)B^{1/2}(|g|^2)(z).$$
\begin{theorem}\label{pplus} Let $f,g\in A^{2}(\D)$. Then $T_fT_g^*$ is bounded on $A^{2}(\D)$ if and only if the 
operator $P^+_{f,g}$ defined by $$P^+_{f,g}u(z)=|f(z)|\int_\D\frac{|g(\z)|u(\z)}{|1-\overline{\z}z|^2}dA(\z)$$
is bounded on $L^2(\D)$.
\end{theorem}
For the proof of the theorem, we need  some preliminary estimates and begin with a completely elementary lemma which will play the key role in our argument.
\begin{lemma}\label{kernelest} For $z,\z\in\D$ we  have
\begin{align*}\frac1{|1-\overline{\z}z|^2}&= -\frac{\overline{\z}z}{(1-\overline{\z}z)^2} +
\frac{1-|z\z|^2}{(1-\overline{\z}z)|1-\overline{\z}z|^2} 
\\&=
-\text{\rm Re}\frac{\overline{\z}z}{(1-\overline{\z}z)^2}+\frac{1-|z\z|^2}{2|1-\overline{\z}z|^2}
+\frac{(1-|z\z |^2)^2}{2|1-\overline{\z}z|^4}\,.\end{align*}
\end{lemma}

\begin{proof} Let $w=\overline{\z}z\in \D$, and note that 
\begin{align*}
&\frac1{(1-w)^2} + \frac1{|1-w|^2}=\frac2{(1-w)}\text{Re}\frac1{(1-w)}\\&
 = \frac1{(1-w)}+ \frac1{(1-w)}\text{Re}
\frac{1+w}{1-w}\,,\end{align*}
and the first identity follows from $\text{Re}\frac{1+w}{1-w}=\frac{1-|w|^2}{|1-w|^2}$. For the second, we just take the real part on both sides of the first and 
use $\text{Re}\frac{1}{1-w}=\frac1{2}+\frac{1-|w|^2}{2|1-w|^2}$.
\end{proof}
The next two lemmas deal with estimates for integral operators whose kernels are involved in the identities above.
\begin{lemma} \label{p1est} For  $f,g \in A^{2}(\D)$, $u\in L^2(\D)$ and $z\in \D$ let 
\begin{align*} &C^{1}_{f, g}u(z)=|f(z)|\int_\D|g|u(\z)\frac{1-|z|^2}{|1-\overline{\z}z|^2}dA(\z)\,,\\&
C^{2}_{f, g}u(z)=|f(z)|\int_\D|g|u(\z)\frac{(1-|z|^2)(1-|\z|^2)}{|1-\overline{\z}z|^4}dA(\z)\,,\\&
C^{3}_{f, g}u(z)=|f(z)|\int_\D|g|u(\z)\frac{1-|\z|^2}{(1-\overline{\z}z)|1-\overline{\z}z|^2}dA(\z)\,,\\&
C^{4}_{f, g}u(z)=|f(z)|\int_\D|g|u(\z)\frac{(1-|\z|^2)^2}{|1-\overline{\z}z|^4}dA(\z)\,.\end{align*}
Then for $j=1,2$
$$ \|C_ju\|_2\lesssim b_{f,g}\|u\|_2\,.$$
Moreover,  for any measurable set $E\subset\D$  
$$ \|C^{3}_{f, g}g1_E\|_2\lesssim \|P^+_{f,g}g1_E\|_2^{1/2}b^{1/2}_{f,g}\|g1_E\|_2^{1/2}\,,$$
and 
$$ \|C^{4}_{f, g}g1_E\|_2\lesssim b_{f,g}\|g1_E\|_2\,.$$
\end{lemma}
\begin{proof} By the Cauchy-Schwartz inequality we have 
$$ |C^{1}_{f, g}u(z)| \le|f(z)|B^{1/2}(|g|^2)(z)\|u\|_2\le b_{f,g}\|u\|_2\,.$$ Similarly,
$$|C^{2}_{f, g}u(z)|\le|f(z)| B^{1/2}(|g|^2)(z)\left(\int_\D \frac{|u(\z)|^2(1-|\z|^2)^2}{|1-\overline{\z}z|^4}dA(\z)\right)^{1/2}\,,$$
so that 
$$\|C^{2}_{f, g}u\|_2^2\le b^2_{f,g}\int_\D\left(\int_\D \frac{|u(\z)|^2(1-|\z|^2)^2}{|1-\overline{\z}z|^4}dA(\z)\right)dA(z)\lesssim  b^2_{f,g}\|u\|^2_2$$
by a standard estimate for integrals (see for example page 10 in \cite{MR1758653}). Another application of  the Cauchy-Schwartz inequality shows that it will suffice to prove the estimate for $C^{4}_{f, g}$ since 
$$|C^{3}_{f, g}u(z)|\le (P^+_{f,g}u(z))^{1/2}C^{1/2}_4u(z)\,.$$ This follows essentially the argument in  \cite{Aleman20122359} (proof of Lemma 3.1). Use the inequality $|1-\overline{\l}w|\le 
|1-\overline{z}w|+|1-\overline{\l}z|$ to obtain
\begin{align*}
&\|C^{4}_{f, g}u\|_2^2\le\\& \int_\D\int_\D\int_\D|f(z)|^2\frac{|gu(\l)||gu(w)|(1-|\l|^2)^2(1-|w|^2)^2}{|1-\overline{z}w|^4|1-\overline{z}\l|^4}
dA(\l)dA(w)dA(z)\\&
\lesssim \int_\D\int_\D\int_\D|f(z)|^2\frac{|gu(\l)||gu(w)|(1-|\l|^2)^2(1-|w|^2)^2}{|1-\overline{\l}w|^4|1-\overline{z}\l|^4}
dA(\l)dA(w)dA(z)\\&
+\int_\D\int_\D\int_\D|f(z)|^2\frac{|gu(\l)||gu(w)|(1-|\l|^2)^2(1-|w|^2)^2}{|1-\overline{\l}w|^4|1-\overline{z}w|^4}
dA(\l)dA(w)dA(z)\\& = 2 \int_\D\int_\D\int_\D|f(z)|^2\frac{|gu(\l)||gu(w)|(1-|\l|^2)^2(1-|w|^2)^2}{|1-\overline{\l}w|^4|1-\overline{z}w|^4}
dA(\l)dA(w)dA(z)\\&=2\int_\D\int_\D B(|f|^2)(w)|gu(w)|\frac{|gu(\l)|(1-|\l|^2)^2}{|1-\overline{\l}w|^4}
dA(\l)dA(w)\,.
\end{align*}
 When $u=gv$ the inequality $|g|^2B(|f|^2)\le b_{f,g}^2$ together with the standard  estimate for integrals mentioned above yield
 $$\|C^{4}_{f, g}gv\|^2\lesssim b_{f,g}^2\|v\|_\infty\|g^2v\|_1\,,$$
and choosing $v=1_E$ the result follows. 
\end{proof} 
In what follows we shall use the well known complex differential operators $\partial=\frac{\partial}{\partial z},~
\overline{\partial}=\frac{\partial}{\partial \overline{z}}$. Let us also note that if $f,g \in A^{2}(\D)$ and $T_fT_g^*$ is bounded on $A^{2}(\D)$ then it is bounded on $L^2(\D)$ with the same norm, since for $u\in L^2(\D)$ we have $$T_fT_g^*u=T_fT_g^*Pu\,,$$
where $P$ is the Bergman projection.
\begin{lemma}\label{partialderiv}
 For  $f,g \in A^{2}(\D)$ with $f(0)=0$, $u\in L^2(\D)$ and $z\in \D$ let 
 $$Ru(z)=\int_\D\frac{|gu(\z)|}{|1-\overline{\z}z|^2}dA(\z)\,,$$
 and let $$Tu(z)=\frac1{z}Ru(z)- (1-|z|^2)\partial Ru(z)\,,\quad Su(z)=Ru(z)-\frac{(1-|z|^2)^2}{2}\partial\overline{\partial}Ru(z)\,. $$
Then
$$\|fTg1_E\|_2\lesssim (\|T_fT_g^*\|+b_{f,g})\|g1_E\|_2+ \|P^+_{f,g}g1_E\|_2^{1/2}b^{1/2}_{f,g}\|g1_E\|_2^{1/2}\,,$$
and 
$$\|fSg1_E\|_2 \lesssim (\|T_fT_g^*\|+b_{f,g})\|g1_E\|_2\,,$$
for all measurable sets $E\subset\D$.
\end{lemma}
\begin{proof}
Rewrite the first identity in Lemma \ref{kernelest} as 
\begin{align*}\frac1{z|1-\overline{\z}z|^2}&= -\frac{\overline{\z}}{(1-\overline{\z}z)^2} +
\frac{1-|z\z|^2}{z|1-\overline{\z}z|^2}+
\frac{(1-|z\z|^2)\overline{\z}}{(1-\overline{\z}z)|1-\overline{\z}z|^2}\\& =  -\frac{\overline{\z}}{(1-\overline{\z}z)^2} +
\frac{1-|z|^2}{z|1-\overline{\z}z|^2}+ \frac{(1-|\z|^2)\overline{z}}{|1-\overline{\z}z|^2}\\&+
\frac{(1-|\z|^2)\overline{\z}|z|^2}{(1-\overline{\z}z)|1-\overline{\z}z|^2}+\frac{(1-|z|^2)\overline{\z}}{(1-\overline{\z}z)|1-\overline{\z}z|^2}
\end{align*}
 Let $M$ be the operator of multiplication by the independent variable on $L^2(\D),~Mv(z)=zv(z)$. It is obvious that $M$ is a bounded operator on $L^{2}(\D)$. As it turns out it also satisfies a bound from below in some cases, namely,
 \begin{equation}\label{subharmonic}
 \|v\|_{L^{2}(r\D)}\lesssim \|Mv\|_{L^{2}(r\D)}\,,\end{equation}
 and
  \begin{equation}\label{wsubharmonic}
  \|v\|_{L^{2}_{(1-|z|^{2})^{2}}(r\D)}\lesssim\|Mv\|_{L^{2}_{(1-|z|^{2})^{2}}(r\D)}\,,\end{equation}
 valid for all subharmonic functions $v$ in $\D$ and all $0<r\leq 1$. These estimates can be easily deduced from the subharmonicity of $v$.
For  a measurable function $h$ on $\D$ let $\phi_h(z)=\overline{h(z)}/|h(z)|$, when $h(z)\ne 0$, and $\phi_h(z)=1$ otherwise, and denote by $U_h$ the unitary operator of multiplication by $\phi_h$ on $L^2(\D)$. 
Multiply both sides by $|gu(\z)|$, integrate on $\D$ w.r.t. $dA(\z)$, and note that
$$\partial Ru(z)=\int_\D\frac{|gu(\z)|\overline{\z}}{(1-\overline{\z}z)|1-\overline{\z}z|^2}dA(\xi)\,.$$
Using the above notations
we obtain 
\begin{align*}|f|Tu(z)&=\frac1{z}|f|Ru(z)-(1-|z|^2)|f|\partial Ru(z)\\&= -(U_fT_fT_g^*U^*_gM^*|u|)(z)+ \frac1{z}(C^{1}_{f, g}|u|(z))+M^*(C^{1}_{g, f})^*|u|(z)\\&+(M^*MC^{3}_{f, g}M^*|u|)(z)\,.\end{align*}
If we let $u=g1_E$ then the first estimate in the statement follows directly by  Lemma \ref{p1est} together with the fact that $b_{f,g}=b_{g,f}$ and \eqref{wsubharmonic}.
The proof of  the second estimate is similar. We rewrite the second  identity in Lemma \ref{kernelest} as
\begin{align*}\frac1{|1-\overline{\z}z|^2}&=
-\text{\rm Re}\frac{\overline{\z}z}{(1-\overline{\z}z)^2}+\frac{1-|z|^2}{2|1-\overline{\z}z|^2}+\frac{|z|^2(1-|\z|^2)}{2|1-\overline{\z}z|^2}
+\frac{(1-|\z |^2)^2}{2|1-\overline{\z}z|^4}\\&
+\frac{(1-|z|^2)(1-|\z |^2)(|\z|^2+|\z|^2)}{2|1-\overline{\z}z|^4}+\frac{(1-|z |^2)^2|\z|^4}{2|1-\overline{\z}z|^4}
\,,\end{align*}
multiply both sides by $|gu(\z)|$, integrate on $\D$ w.r.t. $dA(\z)$, and note that
$$(1-|z|^2)^2\partial\overline{\partial}Ru(z)=(1-|z |^2)^2\int_\D\frac{|gu(\z)||\z|^2}{|1-\overline{\z}z|^4}dA(\z)\,.$$
Thus with the notations above we have
\begin{align*}&|f|Su(z)= -\text{Re}(MU_fT_fT_g^*U^*_gM^*|u|)(z)+\frac1{2}C^{1}_{f, g}|u|(z)+\frac1{2}M^*M(C^{1}_{f, g})^*|u|(z)\\&+\frac1{2}C^{4}_{f, g}|u|(z)
+\frac1{2}(C^{2}_{f, g}(2M^{*} M)|u|)(z)+ \frac1{2}((I-M^{*}M) C^{2}_{f, g} M^{*} M|u|)(z)\,. \end{align*}
If we let $u=g1_E$ then  the result follows  by  another application of Lemma \ref{p1est}.
\end{proof}
With the lemmas in hand we can now proceed to the proof of our theorem.\\

{\it Proof of Theorem \ref{pplus}.} 
Of course, the interesting part is to prove the boundedness of $P_{f,g}^+$ under the assumption that $T_fT_g^*$ is bounded.
By Corollary \ref{c.twmBp} it suffices to show that
\begin{equation}\label{toshow} \|P_{f,g}^{+}|g|1_{Q_{I}}\|_2\lesssim\||g|1_{Q_{I}}\|_2\,,\quad \|P_{g,f}^{+}|f|1_{Q_{I}}\|_2\lesssim\||f|1_{Q_{I}}\|_2\,,
\end{equation}
for all Carleson boxes $Q_{I}$ with $I$ an interval in $\mathbb T$. To this end, let us assume first that $f(0)=0$, and that  $u\in L^2(\D)$ is compactly supported. We shall focus our attention on the function 
\begin{equation}\label{mainobj} \E(z)=(1-|z|^2)^2\partial\overline{\partial}(P_{f,g}^+|u|)^2(z)-\text{Re}\left(\frac{(1-|z|^2)}{z}\overline{\partial}(P_{f,g}^+|u|)^2(z)\right)\end{equation}
The standard growth estimate  for Bergman space functions (see page 54 in \cite{MR1758653}) shows that under our assumptions we can apply Stokes' formula and one of the Green's identities
to conclude that 
\begin{align*}
\int_\D\E(z)dA(z)&=\int_\D (P_{f,g}^+|u|)^2(z)(\partial\overline{\partial}(1-|z|^2)^2)+\text{Re}\left(\frac{1}{z}\overline{\partial}(1-|z|^2)\right)dA(z)\\&
=\int_\D (P_{f,g}^+|u|)^2(z)(4|z|^2-3)dA(z)\,,\end{align*}
so that 
\begin{equation}\label{stokes}
\int_\D\E(z)dA(z)\le \int_\D (P_{f,g}^+|u|)^2(z)dA(z)\,.\end{equation}
With the notation in Lemma \ref{partialderiv} we have  $(P_{f,g}^+|u|)^2=|f|^2R^2u$, and a direct computation gives 
$$\overline{\partial}(P_{f,g}^+|u|)^2=\overline{f'}fR^2u +  2|f|^2 Ru \overline{\partial} Ru\,.$$
In the formulas below we will commit a convenient abuse of notation and  write $z$ also for the  identity function on $\D$. Use  Lemma \ref{partialderiv} to obtain
$$\frac1{z}\overline{\partial}(P_{f,g}^+|u|)^2(z)=\frac1{z}\overline{f'}(z)f(z)R^2u(z) +  2(1-|z|^2)|f|^2(z) \partial Ru (z)\overline{\partial} Ru(z) +2|f|^2(z)Tu(z)\overline{\partial} Ru(z)\,.$$
Obviously, $(1-|z|^2)|\overline{\partial} Ru|\le 2Ru$, and $ \partial Ru \overline{\partial} Ru \ge 0$, hence 
\begin{align}\label{d-barpart}
&\text{\rm Re}\left(\frac{(1-|z|^2)}{z}\overline{\partial}(P_{f,g}^+|u|)^2(z) \right)\le (1-|z|^2)\text{Re} \frac1{z}\overline{f'}(z)f(z)R^2u(z)\\&\nonumber
+ 2(1-|z|^2)^2|f|^2\partial Ru \overline{\partial} Ru+ 4|f|^2Ru|Tu|\,. \end{align}
Similarly, we compute
$$\partial\overline{\partial}(P_{f,g}^+|u|)^2= |f'|^2R^2u+4\text{Re}\overline{f'}fRu\partial Ru
 + 2|f|^2\partial Ru \overline{\partial} Ru + 2|f|^2 Ru \partial\overline{\partial} Ru\,.$$ 
and apply Lemma  \ref{partialderiv} to obtain
\begin{align}\label{laplacepart} &(1-|z|^2)^2\partial\overline{\partial}(P_{f,g}^+|u|)^2(z)= (1-|z|^2)^2|f'|^2(z)R^2u(z)\\&\nonumber+4(1-|z|^2)\text{Re}\left(\frac{1}{z}\overline{f'}(z)f(z)R^2u(z)\right) + 2(1-|z|^2)^2|f|^2(z)\partial Ru(z) \overline{\partial} Ru(z) +4|f|^2(z)R^2u(z) \\&\nonumber -4 (1-|z|^2)\text{Re} \left (\overline{f'}(z)f(z)Ru(z)Tu(z) \right)-4|f|^2(z)Ru(z)Su(z)\,.
\end{align}
From \eqref{d-barpart} and \eqref{laplacepart} we have 
\begin{align}\label{firstEest}
\E&\ge (1-|z|^2)^2|f'|^2(z)R^2u(z)+3(1-|z|^2)\text{Re}\left(\frac{1}{z}\overline{f'}(z)f(z)R^2u(z)\right) +4|f|^2(z)R^2u(z) \\&\nonumber-4 (1-|z|^2)\text{Re} \left (\overline{f'}(z)f(z)Ru(z)Tu(z) \right)-4|f|^2(z)Ru(z)Su(z) -4|f|^2(z)Ru(z)|Tu(z)|\,.\end{align} Fix $\delta\in (\frac3{4},1)$ and use the inequalities
\begin{align*}&\delta (1-|z|^2)^2|f'|^2(z)R^2u(z)+3(1-|z|^2)\text{Re}\left(\frac{1}{z}\overline{f'}(z)f(z)R^2u(z) \right) \ge -\frac{9}{4\delta|z|^2}|f|^2(z)R^2u(z)\\& = -\frac{9}{4\delta}|f|^2(z)R^2u(z)- \frac{9}{4\delta|z|^2}|f|(z)Ru(z) C^{1}_{f, g}|u|  \,, \\ \end{align*}
$$(1-\delta)(1-|z|^2)^2|f'|^2(z)R^2u(z) -4 (1-|z|^2)\text{Re}\left(\overline{f'}(z)f(z)Ru(z)Tu(z)\right)\ge -\frac4{1-\delta}|f|^2(z)|Tu|^2(z)\,,$$
that come from completing squares to conclude that 
\begin{align*}\E(z)&\ge (4-\frac{9}{4\delta})|f|^2(z)R^2(z)-\frac4{1-\delta}|f|^2(z)|Tu|^2(z)-4|f|^2(z)Ru(z) Su(z)\\& -4|f|^2(z)Ru(z)|Tu|(z)-\frac{9}{4\delta|z|^2}|f|(z)Ru(z) C^{1}_{f, g}|u|(z)\,.\end{align*}
Now recall that $|f|^2R^2u=(P_{f,g}^+|u|)^2$, and use the previous inequality in \eqref{stokes} together with the Cauchy-Schwartz inequality and the estimates \eqref{subharmonic} and \eqref{wsubharmonic} to obtain
\begin{align*}&(3-\frac{9}{4\delta})\|P^+_{f,g}|u|\|_2^2 \le\frac4{1-\delta}\|fTu\|_2^2+4\|P^+_{f,g}|u|\|_2\|fSu\|_2\\& +4\|P^+_{f,g}|u|\|_2\|fTu\|_2+\frac{9k}{4\delta} \|P^+_{f,g}|u|\|_2\|C^{1}_{f, g}|u|\|_2\,,\end{align*}
where $k$ is the constant in \eqref{subharmonic}. Now let $u=g1_E$ for a measurable set $E$ with $\overline{E}\subset\D$. By the last inequality and the lemmas \ref{partialderiv} and \ref{p1est} we have   
\begin{align}&\label{finalest}(3-\frac{9}{4\delta})\|P^+_{f,g}|g|1_E\|_2^2\lesssim (\|T_fT_g^*\|+ b_{f,g})^2\|g1_E\|_2^2\\&\nonumber
+\|P^+_{f,g}|g|1_E\|_2 (\|T_fT_g^*\|+b_{f,g})\|g1_E\|_2 +\|P^+_{f,g}|g|1_E\|_2^{3/2}b^{1/2}_{f,g}    \|g1_E\|^{1/2}_2\,.
 \end{align}
Since $(3-\frac{9}{4\delta})>0$, this immediately implies that 
$$\|P^+_{f,g}|g|1_E\|_2\lesssim (\|T_fT_g^*\|+ b_{f,g})\|g1_E\|_2\,.$$
The assumption that $\overline{E}\subset \D$ is easily removed by an approximation argument and Fatou's lemma, while the assumption $f(0)=0$ can be removed by another use of \eqref{subharmonic}. Finally, the remaining estimate in \eqref{toshow} is obtained by interchanging $f$ and $g$, so that the proof is complete. $\square$


\section{A counterexample to Sarason's conjecture for Bergman space}

\label{sec:counter}

Recall that for $f,g\in L_a^2$, we have denoted by $b_{f,g}$  the supremum of the product of the Berezin transforms of $|f|^2$ and $|g|^2$.
In this section we will prove Theorem \ref{t.conjfails}. The proof requires several steps. We begin with the following notations.
The Dirichlet space $D$ consists of analytic  functions $u$ in $\D$ whose derivative belongs to $L_a^2$, and the norm is defined by $$\|u\|_D^2=|u(0)|^2+\|u'\|_2^2\,.$$
Given $f\in L_a^2$ we denote by
$$\gamma^2(f)=\sup_I \log\frac{2\pi}{|I|}\int_{Q_I}|f|^2dA\,,$$
where the supremum is taken over all arcs $I\subset\T$, and by
$$\delta^2(f)=\sup_{\|u\|_D\le 1}\int_\D|fu|^2dA\,.$$
It is well known and easy to prove that $$\gamma(f)\lesssim\delta(f)\,.$$
The fact that these quantities are not comparable has been discovered by Stegenga \cite{MR0550655} and will play an essential role in our argument.\\
 The next lemma relates these numbers to the boundedness of Toeplitz products and products of Berezin transforms.
\begin{lemma}\label{ourex} Let  $f\in L_a^2$, and let $g$  be  a Lipschitz analytic function in $\D$ with \begin{equation}\label{lowbg}|g(z)|\ge c(1-|z|)\,,\end{equation}
for some constant $c>0 $ and all $z\in \D$.\\
(i)  If $fg\in H^\infty$ and  $\gamma(f)<\infty$ then $b_{f,g} <\infty$.\\
 (ii) If  $T_fT_g^*$ is bounded then $\delta(f)<\infty$.
\end{lemma}

\begin{proof}  (i)  Since $g$ is Lipschitz we have 
\begin{align*} B(|g|^2)(z)&\lesssim |g(z)|^2+(1-|z|^2)^2\int_\D \frac{|g(\zeta)-g(z)|^2}{|1-\bar{\z}z|^4}dA(\z)\\&
\lesssim  |g(z)|^2+(1-|z|^2)^2\int_\D\frac{1}{|1-\bar{\z}z|^2}dA(\z)\\&
\lesssim  |g(z)|^2+(1-|z|^2)^2\log\frac2{1-|z|}\,.\end{align*}
Similarly,
\begin{align*} |g(z)|^2B(|f|^2)(z)&\lesssim  (1-|z|^2)^2\int_\D\frac{|g(\zeta)-g(z)|^2|f(\z)|^2}{|1-\bar{\z}z|^4}dA(\z)
+B(|fg|^2)(z)\\&
\lesssim \|f\|^2_2+\|fg\|_\infty^2\,.\end{align*}
Moreover,  let 
$$A_k(z)=\D\cap\{2^k(1-|z|)\le 
|1-\overline{z}\z| \le 2^{k+1}(1-|z|)\}\,,$$
and note that 
\begin{align*}(1-|z|)^2&\log\frac2{1-|z|}B(|f|^2)(z)\sim\log\frac2{1-|z|}\sum_{2^k(1-|z|)\le 2}2^{-4k}\int_{A_k(z)} |f|^2dA \\&
\le \sum_{2^k(1-|z|)\le 2}(\log\frac2{2^k(1-|z|)}+k)2^{-4k}\int_{A_k(z)}|f|^2dA\,,\end{align*}
Clearly, each set  $A_k(z)$ is contained in a Carleson box of perimeter comparable to $2^k(1-|z|)$, hence
$$(\log\frac2{2^k(1-|z|)}+k)\int_{A_k(z)}|f|^2dA\lesssim \gamma^2(f)+k\|f\|_2^2$$
which implies 
$$\sup_{z\in \D}(1-|z|^2)^2\log\frac2{1-|z|}B(|f|^2)(z)\lesssim \gamma^2(f)+\|f\|_2^2\,.$$
Thus $$b_{f,g}\lesssim \|f\|_2^2+\|fg\|_\infty^2+\gamma^2(f)\,.$$
(ii) Let $$Ru(z)=\int_\D\frac{(1-|\z|)u(\z)}{(1-\bar{\z}z)^2}dA(\z)\,,\quad u\in L_a^2\,, z\in \D\,. $$
It is well known and easy to show that $R$ is a bounded invertible operator from   $L_a^2$ onto the Dirichlet space $D$. Moreover, if  $g$ satisfies \eqref{lowbg} we have the obvious inequality
$$|f(z)Ru(z)|\le P^+_{f,g}|u|(z)\,.$$  
 If   $T_fT_g^*$ is bounded then  by Theorem \ref{pplus} we have $$\|P^+_{f,g}|u|\|_2\lesssim
\|u\|_2\,,$$
for all $u\in L_a^2$, hence, by the above argument $$\|fv\|_2\lesssim \|v\|_D\,,$$
for all $v\in D$, and the proof is complete.
 \end{proof}
 
We now construct a special Lipschitz function $g$ with the property  \eqref{lowbg}.\\
Consider sequences $\alpha=(\alpha_j)$ , where all but finitely many terms are zero, and the remaining ones are  equal to one.
Let $$\lambda_0=1\,,\quad \lambda_j=2^{-2^j}\,, j=1,2\ldots\,.$$
Given a sequence $\alpha$  as above, let 
$$x_\alpha=\sum_{j\ge 1}\alpha_j(1-\lambda_j)\lambda_0\ldots\lambda_{j-1}\,,$$
and let $E_1\subset \mathbb R$ be the closure of the set of points $x_\alpha$. Finally, let $E$ be the preimage  of $E_1$ by the conformal map $\phi(z)=i\frac{1+z}{1-z}$, from the unit disc onto  the upper half-plane. The following lemma is a direct application of a result in \cite{MR538552} and has been suggested to us by Konstantin Dyakonov. 

\begin{lemma}\label{dynkin}   There exists a Lipschitz analytic function $g$ in $\D$ which satisfies \eqref{lowbg} and vanishes on $E\cup\{1\}$.
\end{lemma}

\begin{proof}  We claim that $E$ satisfies the condition $(K)$ in \cite{MR538552}, that is 
$$|I|\lesssim\sup_{z\in I}dist(z,E)\,,$$
for all arcs $I\subset \T$.  If we assume the claim, then by Theorem 4 in \cite{MR538552} there exists an outer function $w_{1/2}$ in $\D$ such that $$|w_{1/2}(z)|\sim dist(z,E)^{1/2}\,,\quad |w'_{1/2}(z)|\lesssim dist(z,E)^{-1/2}\,,\quad z\in \D\,.$$ 
If we set $g_1=w_{1/2}^2$  then clearly,
$$(1-|z|)\le dist(z,E)\sim |g_1(z)|\,,\quad |g_1'(z)|=2|w'_{1/2}w_{1/2}(z)|\lesssim 1\,, \quad z\in \D\,,$$
i.e. $g_1$ is Lipschitz, vanishes on $E$ and satisfies \eqref{lowbg}. Since $1\notin E$ it follows that $g(z)=(1-z)g_1(z)$ has the properties required in the statement. \\
To verify the claim,  note first that since $\phi^{-1}$ is analytic and one-to-one in a neighborhood of $E_1$, it will suffice to verify the condition $(K)$ for $E_1$ and all intervals $I\subset\mathbb R$. To this end, we use the obvious   inequality
\begin{equation}\label{tau} \tau =\sup_{j\ge 1}\frac{\sum_{m>j}(1-\lambda_m)\lambda_0\ldots\lambda_{m-1}}{(1-\lambda_j)\lambda_0\ldots\lambda_{j-1}}<\frac1{2}\,.\end{equation}   
Indeed, 
$$\frac{\sum_{m>j}(1-\lambda_m)\lambda_0\ldots\lambda_{m-1}}{(1-\lambda_j)\lambda_0\ldots\lambda_{j-1}}<
\frac{\lambda_j}{1-\lambda_j}\left(1+\sum_{m>j+1}2^{-m}\right)<\frac{1}{3}(1+2^{-2})\,.$$
In particular, \eqref{tau} shows that if $x_\alpha<x_\beta$ then there exists $j\ge 1$ such that $\beta_j-\alpha_j=1$, and $\alpha_m=\beta_m$ for $m< j$.  Moreover, in this case  we have that 
\begin{equation}\label{Kstep} \left|\frac{x_\alpha+x_\beta}{2}-x_{\alpha'}\right|>k(x_\beta-x_\alpha)\,,
\end{equation}
for some $k>0$ independent of $\alpha,\alpha',\beta$. To see this note that the inequality holds with $k=\frac1{2}$,  when $x_{\alpha'}$ lies outside $(x_\alpha,x_\beta)$. When $x_{\alpha'} $ lies inside this interval, with $j$  given above we have by \eqref{tau}
\begin{align*}\left|\frac{x_\alpha+x_\beta}{2}-x_{\alpha'}\right|&>\frac1{2}(1-\lambda_j)\lambda_0\ldots\lambda_{j-1}-
\sum_{m>j}(1-\lambda_m)\lambda_0\ldots\lambda_{m-1}\\&>(\frac1{2}-\tau)(1-\lambda_j)\lambda_0\ldots\lambda_{j-1}\\&>
\frac{\frac1{2}-\tau}{1+\tau}(x_\beta-x_\alpha)\,.\end{align*}
Finally, \eqref{Kstep} immediately implies $(K)$. If $(a,b)$ is any interval and $$dist(a,E_1),dist(b,E_1)>\frac1{3}(b-a)\,,$$
then $(K)$ holds with constant $\frac1{3}$. If 
$$dist(a,E_1),dist(b,E_1)<\frac1{3}(b-a)\,,$$
we can find $x_\alpha,x_\beta\in E_1$ such that
$$|x_\alpha-a|\,,|x_\beta-b|<\frac1{3}(b-a),$$
and then $\frac{x_\alpha+x_\beta}{2}\in (a,b)$, and $x_\beta-x_\alpha>\frac1{3}(b-a)$, so that the result follows from above.

\end{proof} 

\begin{Pf} {\it Theorem \ref{t.conjfails}.}
Assume the contrary. Fix  a function  $g$  as in Lemma \ref{dynkin}, this function clearly belongs to $ L_a^2$, and consider the space $X_g$ consisting of functions $f\in L_a^2$ with $$\|f\|=\|fg\|_\infty+\gamma(f)<\infty\,.$$
It is obviously a Banach space. By Lemma \ref{ourex} (i) we have $b_{f,g}<\infty$
for all $f\in X_g$, by assumption this implies that $T_fT_g^*$ is bounded on $L_a^2$ whenever $f\in X_g$, and  finally, by
 Lemma \ref{ourex} (ii) we obtain that $\delta(f)<\infty,~f\in X_g$. Since the space of functions $u\in L_a^2$ with $\delta(u)<\infty$ and norm given by $u\to\delta(u)$ is at its turn a Banach space, we can apply the colsed graph theorem to conclude that there exists $c>0$ such that \begin{equation}\label{contrad} \delta(f)\le c(\|fg\|_\infty+\gamma(f))\,,
 \end{equation}
 for all $f\in X_g$. We will show that this leads to a contradiction.\\
Recall that $\phi(z)=i\frac{1+z}{1-z}$ is the conformal map from the unit disc onto  the upper half-plane. With the notations preceding Lemma \ref{dynkin}, D. Stegenga (\cite{MR0550655}, p. 136) has constructed a sequence $(f_n)$ in $L_a^2$ of the form 
$$f_n(z)=2^{-n/2}p_n\sum_{k=1}^{2^n} \frac{1}{(\phi(z)-z_{nk})^2}\phi'(z)\,,$$
where $p_n=\lambda_0\ldots\lambda_n$, $\text{Im}z_{nk}<0$, with 
\begin{equation}\label{zenkay} dist(z_{nk},E_1) =-\text{Im}z_{nk}\sim p_n\end{equation} such that
\begin{equation}\label{stegenga1}\limsup_{n\to\infty}\gamma(f_n)<\infty\,,\end{equation}  
\begin{equation}\label{stegenga2}\lim_{n\to\infty}\delta(f_n)=\infty\,,\end{equation}
\begin{equation}\label{stegenga3}\sup_{z\in \D\atop n\in \N}2^{-n/2}p_n\sum_{k=1}^{2^n} \frac{1}{|\phi(z)-z_{nk}|}<\infty\,.\end{equation}
A simple calculation gives 
$$\frac{\phi'(z)}{\phi(z)-z_{nk}}=\frac{2i}{(z_{nk}+i)(1-z)(z-\phi^{-1}(z_{nk}))}\,,$$
and by \eqref{zenkay} there exist points $\zeta_{nk}\in E$ with $$|\phi^{-1}(z_{nk})-\zeta_{nk}|\sim |\phi^{-1}(z_{nk})|-1=dist(\phi^{-1}(z_{nk}),\T)\,.$$
Since $g$ has the properties in  Lemma \ref{dynkin},that is, it is Lipschitz and vanishes at $1,~\zeta_{nk}$,  it follows immediately that 
$$\left|\frac{g(z)\phi'(z)}{\phi(z)-z_{nk}}\right|\le C\,,$$
for some absolute constant $C>0$, all $k,n\in \N$ with $1\le k\le 2^n$, and all $z\in \D$. From \eqref{stegenga3} we have that $f_ng$ are uniformly bounded in $H^\infty$. Thus  by \eqref{stegenga1} we have that $(f_n)$ is a bounded sequence in $X_g$, hence \eqref{stegenga2} and \eqref{contrad} yield a contradiction which proves the theorem. 
\end{Pf}


\section{The class $B_\infty$ and sharp estimates in terms of the B\'ekoll\'e constant}

\label{sec:appl}

In our last section we include an application of the two weight result for the maximal Bergman projection, namely we obtain sharp B\'ekoll\'e estimates by establishing sharp estimates for the testing conditions \eqref{e.testpbeta} and \eqref{e.testpbetad}. We provide sharper estimates than the ones discussed by Pott and Reguera in \cite{1210.1108}.

\subsection{The class $B_\infty$}
Following B\'ekoll\'e and Bonami \cite{MR497663}, we say that a weight, i.e., a measurable positive function $w$,  
belongs to the class $B_p$ for $1<p<\infty$, if and only if
\begin{equation}
\label{e.bekolle}
B_{p}(w):=\sup_{\substack{I \text{ interval}\\ I\subset \T}}\left(\frac{1}{|Q_{I}|}\int_{Q_{I}}wdA\right) \left(\frac{1}{|Q_{I}|}\int_{Q_{I}}w^{1-p'}dA\right)^{p-1}<\infty
\end{equation}

\begin{definition}
We say that a  weight $w$ belongs to the class $B_{\infty}$, if and only if
\begin{equation}
\label{e.binfinity}
B_{\infty}(w):= \sup_{\substack{I \text{ interval}\\ I\subset \T}}\frac{1}{w(Q_{I})}\int_{Q_{I}} M(w1_{Q_{I}})  < \infty,
\end{equation}
where $M$ stands for the Hardy-Littlewood maximal function over Carleson cubes. 
\end{definition}
This definition of $B_{\infty}$ is motivated by the Muckenhoupt version  $A_{\infty}$ described by Wilson in \cites{MR972707, MR883661, MR2359017}. 
This $A_\infty$ definition appears in the recent works of Lerner \cite{MR2770437}, Hyt\"onen and P\'erez \cite{1103.5562} and Hyt\"onen and Lacey \cite{1106.4797} among others,
where it is used to find sharp estimates in terms of the Muckenhoupt $A_{p}$ and $A_{\infty}$ constants.

In particular,  $B_\infty$ contains any of the classes $B_p$:
\begin{proposition}
Let $w$ be a weight and $1 < p < \infty$. Then
$$
      B_\infty(w) \le B_p(w)^{\frac{1}{p'-1}}.
$$
\end{proposition}
\proof
Let $w \in B_p$ and recall that $B_p(w) = B_{p'}(w')$, where $w' = w^{1-p'}$. Hence for any Carleson cube $Q_I$,
\begin{eqnarray*}
\int_{Q_I} M(1_{Q_I}w) & \le &     \left( \int_{Q_I} M(1_{Q_I}w)^{p'} w'   \right)^{1/p'}                 \left( \int_{Q_I}     w     \right)^\frac{1}{p}   \\
&\le &     \| M(w \cdot) \|_{L^{p'}(w) \to L^{p'}(w')}  w(Q_I)^{1/p'} w(Q_I)^{1/p} \\
& =   & \|  M \|_{L^{p'}(w') \to L^{p'}(w')}  w(Q_I) \le B_p(w)^{\frac{1}{p'-1}} w(Q_I),
\end{eqnarray*}
where we have used the estimate (4.7) from \cite{1210.1108} for the maximal function in the last line.
\qed

\subsection{The sharp estimate}

The main result in this section is the following:

\begin{theorem}
\label{t.bekolle}
Let $w\in B_{2}$ be a Bekoll\'e weight with constant $B_{2}(w)$ and let $P^{+}$ be the positive Bergman projection. Then
\begin{equation}
\label{e.sharpbeko}
\norm P^{+}f. L^{2}(w).\leq C B_{2}(w)^{1/2}(B_{\infty}(w)^{1/2}+ B_{\infty}(w^{-1})^{1/2}) \norm f. L^{2}(w).,
\end{equation}
with $C$ independent of the weight $w$.
\end{theorem}

\begin{corollary}
The same result holds for the Bergman projection $P_B$.
\end{corollary}

The method of proof will be as follows. We will consider the dyadic operators $P^{\beta}$ and use Theorem \ref{c.twodB} to obtain 
the sharp bound in the B\'ekoll\'e constants, which will be independent of the choice of the grid. An averaging operation will now yield the desired result.

The following lemma is known in the case of Muckenhoupt weights if the collection of cubes appearing in the sum is sparse, this can be found in \cite{1103.5562}. In our case, the lemma reads as follows. 

\begin{lemma}\label{l.Ainfty}
Let $\sigma\in B_{\infty}$, then
\begin{equation}
\label{e.Ainfty}
\sum_{K:\, K \subset I}\sigma(Q_{K})\leq 2 B_{\infty}(\sigma) \sigma(Q_{I}).
\end{equation}
\end{lemma}

\begin{proof}
\begin{eqnarray*}
\sum_{K:\, K \subset I}\sigma(Q_{K})&=& \sum_{K:\, K \subset I}\frac{\sigma(Q_{K})}{|Q_{K}|}|Q_{K}|\\
& \leq & 2 \sum_{K:\, K \subset I}\frac{\sigma(Q_{K})}{|Q_{K}|}|T_{K}|\\
& \leq & 2 \sum_{K:\, K \subset I}\int_{T_{K}}M(\sigma 1_{Q_{I}})dm\\
& \leq & 2 B_{\infty}(\sigma)\sigma(Q_{I})
\end{eqnarray*}

\end{proof}

We turn to proving the desired bound for the two testing conditions. 

\begin{proof}[Proof of Theorem \ref{t.bekolle}] We use Theorem \ref{c.twodB} for the weights $w$ and $\sigma=w' = w^{p'-1}$. We only
have to show the appropriate bounds for the test function conditions, and
we will only focus on one of the conditions, as the study of the other is analogous. In what follows, let $I\in \mathcal D^{\beta}$. We want to prove 
\begin{equation}
\label{e.testB}
\norm P^{\beta}_{I,in}(w^{-1} 1_{Q_{I}}).L^{2}(w).^{2}\lesssim B_{2}(w) B_{\infty}(w^{-1})w^{-1}(Q_{I}),
\end{equation}
where the implicit constant does not depend on the chosen grid $D^{\beta}$ or the weight $w$.

\begin{eqnarray*}
\norm P^{\beta}_{I,in}(w^{-1} 1_{Q_{I}}).L^{2}(w).^{2}& = & \int_{Q_I}\left| \sum_{K:\, K\subset I} \langle w^{-1} 1_{Q_{I}}, \frac{1_{Q_{K}}}{|K|} \rangle \frac{1_{Q_{K}}}{|K|}\right|^{2}wdA\\
 &= & \int_{Q_I} \sum_{K:\, K\subset I} \langle w^{-1} 1_{Q_{I}}, \frac{1_{Q_{K}}}{|K|} \rangle^{2} \frac{1_{Q_{K}}}{|K|^{2}}wdA\\
 & & + 
  2\int_{Q_I} \sum_{K':\, K'\subset I} \sum_{K:\, K\subset K'} \langle w^{-1} 1_{Q_{I}}, \frac{1_{Q_{K}}}{|K|} \rangle \langle w^{-1} 1_{Q_{I}}, \frac{1_{Q_{K'}}}{|K'|} \rangle \frac{1_{Q_{K}}}{|K||K'|}wdA\\
 & :=& D + 2OD,\\
 \end{eqnarray*} 

where the terminology for $D$ and $OD$ comes from the diagonal and the off-diagonal. Let us treat each term in turn.

\begin{eqnarray*}
D&=&  \sum_{K:\, K\subset I} \frac{w^{-1}(Q_{K})^{2}}{|K|^{2}}  \frac{w(Q_{K})}{|K|^{2}}\\
 & \leq & B_{2}(w)\sum_{K:\, K\subset I} w^{-1}(Q_{K})\\
 & \leq & B_{2}(w)B_{\infty}(w^{-1}),
\end{eqnarray*}
where the last inequality follows from lemma \ref{l.Ainfty}. The off-diagonal term is equally simple, 

\begin{eqnarray*}
OD&=&  \sum_{K':\, K'\subset I} \sum_{K: \, K\subset K'} \frac{w^{-1}(Q_{K'})}{|K'|} \frac{w^{-1}(Q_{K})}{|K|}  \frac{w(Q_{K})}{|K||K'|}\\
&=& \sum_{K':\, K'\subset I} w^{-1}(Q_{K'})\frac{1}{|K'|^{2}}\sum_{K: \, K\subset K'}  \frac{w^{-1}(Q_{K})}{|K|^{2}} \frac{w(Q_{K})}{|K|^{2}}|K^{2}|\\
 & \leq & B_{2}(w)\sum_{K':\, K'\subset I} w^{-1}(Q_{K'})\frac{1}{|K'|^{2}}\sum_{K: \, K\subset K'}|K^{2}|\\
 & \leq &C B_{2}(w)\sum_{K':\, K'\subset I} w^{-1}(Q_{K'})\\
 & \leq &2C B_{2}(w)B_{\infty}(w^{-1}),
\end{eqnarray*} 
where in the second line we have multiplied and divided by $|K|^{2}$ to use the Bekoll\'e constant. 

\end{proof}

\section{Acknowledgement}
We thank the Crafoord Foundation for the partial support of this research. We also thank Nicola Arcozzi for a helpful discussion on Stegenga's counterexample. Finally, we are deeply indebted to Konstantin Dyakonov, who brought to our attention the work of Dyn{\cprime}kin.



\end{document}